\documentclass[11pt]{article}

\usepackage{mathrsfs}
\usepackage{amsfonts}
\usepackage{amssymb}
\usepackage{amsmath}
\def\disp{\displaystyle}

\def\crr{\cr\noalign{\vskip2mm}}

\newtheorem{theorem}{Theorem}[section]

\numberwithin{equation}{section}

\newtheorem{definition}[theorem]{Definition}
\newtheorem{example}[theorem]{Example}

\newtheorem{remark}[theorem]{Remark}

\newenvironment{proof}[1][Proof]{\textbf{#1.}}{\ \rule{0.5em}{0.5em}}%

\begin{document}
\parindent 9mm
\title{Exponential Stability of a Class of Infinite Dimensional Coupled Systems\thanks{This work was supported by the Natural Science Foundation of Shaanxi Province (grant nos.
2018JM1051, 2014JQ1017), and the Fundamental Research Funds for the Central Universities (Grant no.
xjj2017177). }
}
\author{ Zhan-Dong Mei
\thanks{School of Mathematics and Statistics, Xi'an Jiaotong
University, Xi'an 710049, China.  Email: zhdmei@mail.xjtu.edu.cn.  Telephone number: +86 02982663149.} }


\date{}
\maketitle \thispagestyle{empty}
\begin{abstract}
%
%
This paper is concerned with exponential stability of a class of infinite dimensional coupled systems.
It is proved that under some admissibility conditions, the considered infinite dimensional coupled system
is governed by a $C_0$-semigroup. Furthermore, if both the free subsystems are governed by exponentially stable $C_0$-semigroups, then
so is the coupled system. The results are applied to simplify the proof of semigroup generation and exponential stability for several coupled systems emerged in control theory literatures.

\vspace{0.5cm} 

%
%

\end{abstract}


\section{Introduction and definitions}\label{section1}

As is well-known, in many infinite dimensional linear control systems, the closed-loop systems (or subsystems) under appropriate feedback laws are usually
described by coupled systems as follows
\begin{align}\label{couple}
    \left\{
      \begin{array}{ll}
        \dot{x}(t)=\mathcal{A}x(t)+BCy(t) \\
       \dot{y}(t)=\mathbf{A}y(t)
      \end{array}
    \right.
\end{align}
where $\mathcal{A}$ and $\mathbf{A}$ respectively generate $C_0$-semigroups on Banach spaces $\mathcal{X}$ and $\mathbf{X}$;
$C$ is a linear bounded operator from $D(\mathbf{A})$ to Banach space $U$ with $D(\mathbf{A})$ being Banach space equipped with
graph norm; $B$ is a linear bounded operator from $U$ to the extrapolation space  $\mathcal{X}_{-1}^\mathcal{A}$ of $\mathcal{X}$. By definition \cite{Engel2000}, the
extrapolation space $\mathcal{X}_{-1}^\mathcal{A}$ is the completion of $\mathcal{X}$ under the norm
$\|R(\lambda_0,\mathcal{A})\cdot\|$ with $R(\lambda_0,\mathcal{A})$ the resolvent of $\mathcal{A}$
at $\lambda_0$.  Hence the proofs of the well-posedness (semigroup generation) and exponential (or asymptotic) stability of couple system (\ref{couple}) are essentially important in control theory.

In \cite[(4.42)]{Feng2017b}, \cite[(3.1)]{Guo2007a} and  \cite[(3.1)]{Guo2008a}, the generations of semigroups and exponential stabilities of the coupled systems were proved by virtue of rather complicated Riesz basis approach. In order to verify that system \cite[(10)]{Krstic2008} generates
an exponentially stable $C_0$-semigroup, a quite skillful inner product was introduced and energy multiplier method was used, and the proof is rather length.

All the aforementioned references dealt with the coupled system just case by case, and the proof seems rather
complicated. Motivated by this, in this paper, our objective is to
develop a uniform fame to settle such problem. Concretely, we shall verify in Section 2 that under some considerable conditions system (\ref{couple}) is governed by an exponentially stable $C_0$-semigroup by pure semigroup method. In order to end this section, we introduce the the notions of admissible control and admissible observation, which will be used in next section.
\begin{definition}\cite{Weiss1989,Weiss1989b}
Assume that $X$ and $U$ are Banach spaces, and $A$ generates a $C_0$-semigroup $e^{At}$ on $X$.

(i).  We call $B\in L(U,X_{-1})$ to be admissible for $e^{At}$, if for a (and hence for all)   $t_0>0$   and $u\in L^2_{loc}(0,\infty;U)$, $\int_0^{t_0}e^{A(t-s)}Bu(s)ds\in X$ and
\begin{align*}
    \bigg\|\int_0^{t_0}e^{A(t-s)}Bu(s)ds\bigg\|\leq W_{A,B}(t_0) \bigg(\int_0^{t_0}\|u(s)\|^2ds\bigg)^{\frac{1}{2}},
\end{align*}
where $W_{A,B}:(0,\infty)\rightarrow [0,\infty)$ is a nondescreased function depending on $A$ and $B$.

(ii).  We call $C\in L(D(A),U)$ to be admissible for $e^{At}$, if for any $x\in D(A)$ and a (and hence for all) $t_0>0$, there
holds
\begin{align*}
    \int_0^{t_0}\|Ce^{As}x\|^2ds\leq V_{A,C}^2(t_0)\|x\|^2,
\end{align*}
where $V_{A,C}:(0,\infty)\rightarrow [0,\infty)$ is a nondescreased function depending on $A$ and $C$.
\end{definition}

\section{Main Results}

In this section, we shall present our main results: under some admissibility conditions, system (\ref{couple}) is governed by a $C_0$-semigroup; in addition,
if both the free subsystems are governed by exponentially stable $C_0$-semigroups, so is the coupled system.

\begin{theorem}\label{general}
Assume that $B\in L(U,\mathcal{X}^\mathcal{A}_{-1})$ and $C\in L(D(\mathbf{A}),U)$ are admissible for $e^{\mathcal{A}t}$ and $e^{\mathbf{A}t}$, respectively. Then, the operator $\mathfrak{A}=\left(
                                                    \begin{array}{cc}
                                                      \mathcal{A} & BC \\
                                                      0 & \mathbf{A} \\
                                                    \end{array}
                                                  \right)
$ with domain $D(\mathfrak{A})=\{(f,g)\in \mathcal{X}\times D(\mathbf{A}):\mathcal{A}f+BCg\in \mathcal{X}\}$
generates a $C_0$-semigroup. If, in addition, the semigroup $e^{\mathcal{A}t}$ and $e^{\mathbf{A}t}$ are exponentially
stable, then so is the semigroup $e^{\mathfrak{A}t}$.
\end{theorem}
\begin{proof}\ \
~Since $C$ is admissible for $e^{\mathbf{A}t}$, it follows from \cite{Weiss1989b} that for any $g\in \mathbf{X}$, $e^{\mathbf{A}t}g\in D(C_\Lambda)$, a.e. $t\geq 0$ and $C_\Lambda e^{\mathbf{A}\cdot}g\in L^2_{loc}(0,\infty; U)$, where $\big(C_\Lambda, D(C_\Lambda)\big)$ is an extension of $(\mathbf{A},D(\mathbf{A}))$ defined by
\begin{align*}
    \left\{
      \begin{array}{ll}
        D(C_\Lambda)=\{g\in \mathbf{X}:\mbox{the limt}\lim_{\lambda\rightarrow +\infty}CR(\lambda,\mathbf{A})g\;\mbox{exists}\}, \crr
        C_\Lambda g=\lim_{\lambda\rightarrow +\infty}CR(\lambda,\mathbf{A})g,\; g\in
D(C_\Lambda).
      \end{array}
    \right.
\end{align*}
We define a family of operators $\{\mathrm{S}(t)\}_{t\geq 0}$ by
\begin{align*}
    \mathrm{S}(t)(f,g)^T=\left(
                         \begin{array}{c}
                           e^{\mathcal{A}t}f+\int_0^te^{\mathcal{A}(t-\sigma)}BC_\Lambda e^{\mathbf{A}\sigma}gd\sigma \\
                           e^{\mathbf{A}t}g \\
                         \end{array}
                       \right), \; \forall \;(f,g)\in \mathcal{X}\times \mathbf{X}.
\end{align*}
It is easily seen that $S(\cdot)$ is strongly continuous on $[0,\infty)$. Moreover, since $\mathcal{A}$ and $\mathbf{A}$ are generators of $C_0$-semigroups, for any $\lambda\in\rho(\mathcal{A})\cap\rho(\mathbf{A})$ and $(f,g)\in \mathcal{X}\times \mathbf{X}$ we have
\begin{align}\label{rujieshi}
    \int_0^\infty e^{-\lambda t}\mathrm{S}(t)(f,g)^Tdt=\left(
                                                         \begin{array}{cc}
                                                           R(\lambda,\mathcal{A})f+ R(\lambda,\mathcal{A})BCR(\lambda,\mathbf{A})g \\
                                                           R(\lambda,\mathbf{A})g \\
                                                         \end{array}
                                                       \right)=R(\lambda,\mathfrak{A})(f,g)^T.
\end{align}
By \cite[page 113, Theorem 3.17]{Arendt2001}, $S(t)$ is a $C_0$-semigroup generated by $\mathfrak{A}$.

Next, we shall prove that the semigroup $S(t)=e^{\mathfrak{A}t}$ is exponentially stable provided $e^{\mathcal{A}t}$ and $e^{\mathbf{A}t}$
are exponentially stable. By \cite{Weiss1989b}, the exponential stability of $e^{\mathbf{A}t}$ means that, there exist positive constants $M_\mathbf{A}$, $\omega_\mathbf{A}$ and $N$ such that
\begin{align}\label{N}
 \|e^{\mathbf{A}t}\|\leq M_\mathbf{A} e^{-\omega_\mathbf{A}t} ,\; \;V_{\mathbf{A},C}(t)\leq N,\; \forall\; t>0.
\end{align}
Next,  by \cite{Weiss1989}, the admissibility of $B$ for $e^{\mathcal{A}t}$ implies that
\begin{align}\label{biaoshi}
    \int_0^t e^{\mathcal{A}(t-\sigma)}BC_\Lambda e^{\mathbf{A}\sigma}gd\sigma \in \mathcal{X}, \; \forall \; g\in\mathbf{X}.
\end{align}
Moreover, since the semigroup $e^{\mathcal{A}t}$ is exponentially stable, by \cite{Weiss1989} there exist positive constants $M_\mathcal{A}$, $\omega_\mathcal{A}$ and $K$ such that
\begin{align}\label{K}
  \|e^{\mathcal{A}t}\|\leq M_\mathcal{A}e^{-\omega_\mathcal{A}t},\;\; W_{\mathcal{A},B}(t)\leq K,\; \forall\; t>0.
\end{align}
The combination of (\ref{N}), (\ref{biaoshi}) and (\ref{K}) produces
\begin{align}\label{eee}
   \bigg\| \int_0^t e^{\mathcal{A}(t-\sigma)}BC_\Lambda e^{\mathbf{A}\sigma}gd\sigma \bigg\|
\leq K\bigg(\int_0^t\|C_\Lambda e^{\mathbf{A}\sigma}g\|^2d\sigma\bigg)^{\frac{1}{2}}\leq KN\|g\|, \; \forall \; g\in\mathbf{X}.
\end{align}

Let $0<\gamma<\min\{\omega_{\mathcal{A}},\omega_{\mathbf{A}}\}$ and $(f,g)\in \mathcal{X}\times \mathbf{X}$. By bounded perturbation theorem of $C_0$-semigroups \cite[p.158]{Engel2000},
$\gamma+\mathcal{A}$ and $\gamma+\mathbf{A}$ generate $C_0$-semigroups $e^{(\gamma+\mathcal{A})t}=e^{\gamma t}e^{\mathcal{A}}$ and $e^{(\gamma+\mathbf{A})t}=e^{\gamma t}e^{\mathbf{A}}$, respectively. Moreover, $e^{(\gamma+\mathcal{A})t}$ and $e^{(\gamma+\mathbf{A})t}$ are exponentially stable and the following inequality,  similarly to (\ref{eee}), holds
\begin{align}\label{aaa}
    \bigg\|\int_0^te^{(\gamma+\mathbb{A}_0)(t-\mathcal{A}\sigma)}BC_\Lambda e^{(\gamma+\mathbf{A})\sigma}gd\sigma\bigg\|
\leq K_1N_1 \|g\|,\; \forall g\in\mathbf{X},
\end{align}
where $K_1$ and $N_1$ are positive constants similar to $K$ and $N$, respectively.
We compute
\begin{align*}
    \left\|e^{\gamma t}\left(e^{\mathcal{A}t}f+\int_0^te^{\mathcal{A}(t-\sigma)}BC_\Lambda e^{\mathcal{A}\sigma}gd\sigma\right)\right\|
=&\bigg\|e^{\gamma t}e^{\mathcal{A}t}f+\int_0^te^{\gamma (t-\sigma)}e^{\mathcal{A}(t-\sigma)}BC_\Lambda e^{\gamma \sigma}e^{\mathcal{A}\sigma}gd\sigma\bigg\|\crr
=&\bigg\|e^{(\gamma+\mathcal{A})t}f+\int_0^te^{(\gamma+\mathcal{A})(t-\sigma)}BC_\Lambda e^{(\gamma+\mathcal{A})\sigma}gd\sigma\bigg\|\crr
\leq &M_{\mathcal{A}}e^{-(\omega_{\mathcal{A}}-\gamma) t}\|f\|+K_1N_1\|g\|,
\end{align*}
which implies that
\begin{align*}
    \left\|e^{\mathcal{A}t}f+\int_0^te^{\mathcal{A}(t-\sigma)}BC_\Lambda e^{\mathbf{A}\sigma}gd\sigma\right\|
\leq \max\{M_{\mathcal{A}},K_1N_1\}e^{-\gamma t}(\|f\|+\|g\|).
\end{align*}
Therefore,
\begin{align*}
    \|\mathrm{S}(t)(f,g)\|\leq & M_{\mathbf{A}}e^{-\omega_{\mathbf{A}}t}\|g\|+\max\{M_{\mathcal{A}},K_1N_1\}e^{-\gamma t}(\|f\|+\|g\|)\crr
\leq & \max\{M_{\mathcal{A}}+M_{\mathbf{A}},K_1N_1+M_{\mathbf{A}}\}e^{-\gamma t}(\|f\|+\|g\|),
\end{align*}
which implies that $S(\cdot):[0,\infty)\rightarrow L(\mathcal{X}\times \mathbf{X})$ is exponentially stable. The proof is completed.
\end{proof}

\begin{remark}
We have to mention that the semigroup generation part of Theorem \ref{general} can also be obtained by the feedback theory of regular linear system \cite{Weiss1994}. The core idea of the proof: define $A_{ext} = diag ({\mathcal A}, A)$, $B_{ext} = ( B, 0)^T$, and  $C_{ext}= (0, C)$, and verify that $(A_{ext},B_{ext},C_{ext})$ forms a regular linear system with admissible feedback operator $I$, then $\mathfrak{A}$ is the closed-loop system and thereby generates a $C_0$-semigroup. In our Theorem \ref{general}, we use a direct and simple approach, pure semigroup method, other than feedback theory of regular linear system, and therefore the terms ``regular linear system" and ``admissible  feedback operator" are not involved.
\end{remark}

 The importance of Theorem \ref{general} lies in that it supplies us with a uniform frame to prove semigroup generation and exponential stability for many coupled partial differential equations, since Riesz basis approach as well as energy multiplier approach to coupled partial differential equations seems rather complicated. In the rest of this section, we shall present four examples for the applications of Theorem \ref{general}.

\begin{example}\label{exam1}
 In \cite[(3.1)]{Guo2008a}, the closed-loop system is governed  by
\begin{align}\label{2008a}
    \left\{\begin{array}{l}
\widehat{w}_{tt}(x,t)+\widehat{w}_{xxxx}(x,t)=0,\;\; x\in (0,1), \; t>0, \\
\widehat{w}(0,t)=\widehat{w}_{xx}(1,t)=0, \widehat{w}_{xxx}(1,t)=\gamma \widehat{w}_t(1,t)\; \;   t\ge 0,\\
\widehat{w}_{xx}(0,t)=c_2 \widehat{w}_{xt}(0,t)+c_3 \widehat{w}_{x}(0,t)+w_{xx}(0,t), \;\;  t\ge 0, \\
w_{tt}(x,t)+w_{xxxx}(x,t)=0,\;\; x\in (0,1), \; t>0, \\
w(0,t)=w_x(0,t)=w_{xx}(1,t)=0, \; \;   t\ge 0,\\
w_{xxx}(1,t)=c_1 \widehat{w}_t(1,t), \;\;  t\ge 0, \\
\end{array}\right.
\end{align}
which is equivalent to
\begin{align}\label{2008b}
    \left\{\begin{array}{l}
w_{tt}(x,t)+w_{xxxx}(x,t)=0,\;\; x\in (0,1), \; t>0, \\
w(0,t)=w_x(0,t)=w_{xx}(1,t)=0, \; \;   t\ge 0,\\
w_{xxx}(1,t)=c_1 w_t(1,t)+c_1 \varepsilon_t(1,t), \;\;  t\ge 0, \\
\varepsilon_{tt}(x,t)+\varepsilon_{xxxx}(x,t)=0,\;\; x\in (0,1), \; t>0, \\
\varepsilon(0,t)=\varepsilon_{xx}(1,t)=0, \varepsilon_{xxx}(1,t)=0\; \;   t\ge 0,\\
\varepsilon_{xx}(0,t)=c_2 \varepsilon_{xt}(0,t)+c_3 \varepsilon_{x}(0,t), \;\;  t\ge 0, \\
\end{array}\right.
\end{align}
where $\varepsilon (x,t)=\widehat{w}(x,t)-w(x,t)$, $c_1,c_2,c_3>0$.
Let $\mathbb{H} = H_E^2(0,1)\times L^2(0,1),  H_E^2(0,1)=\{f\in H^2(0, 1)| f(0) =f'(0) = 0\},$
with the inner product induced norm $\|(f,g)\|^2_\mathbb{H} = \int_0^1[|f''(x)|^2+|g(x)|^2]dx$
for any $(f,g)\in \mathbb{H}.$
Define
$$
\left\{\begin{array}{l}
\mathcal{A}(f,g)=(g,-f^{(4)}), \crr\disp
 D(\mathcal{A})=\{(f,g)\in \mathbb{H}| f''(1)=0,f'''(1)=c_1 g(1)\};
 \end{array}\right.
$$
and
$$
\left\{\begin{array}{l} \mathbf{A}(f,g)=(g,-f^{(4)}), \crr\disp
 D(\mathbf{A})=\{(f,g)\in \mathbf{X}|\mathbf{A}(f,g)\in \mathbf{X}, f''(1)=f'''(1)=0,f''(0)=c_2 g'(0)+c_3 f'(0)\}
 \end{array}\right.
$$
 where $\mathbf{X}=H^2_F(0,1)\times L^2(0,1)$, $H^2_F(0,1)=\{f\in H^2(0,1):f(0)=0\}$ with
inner product induce norm $\|(f,g)\|^2_\mathbf{X} = \int_0^1[|f''(x)|^2+|g(x)|^2]dx+c_3 |f'(0)|^2$
for any $(f,g)\in \mathbf{X}.$ Let
$$
B_1=(0,\delta(x-1))^\top \hbox{  and } C_1=(0,c_1 \delta(x-1))
$$
with $\delta$ being Dirac operator.
It is well known that both $\mathcal{A}$ and $\mathbf{A}$ are generators of exponentially stable $C_0$-semigroups on $\mathbb{H}$ and $\mathbf{X}$, respectively. It is easily seen that $B_1$ is admissible for $e^{\mathcal{A}t}$.
By \cite[Theorem 2.1]{Guo2008a}, $C_1$ is admissible for $e^{\mathbf{A}t}$.
Observe that the system operator corresponding to (\ref{2008b}) is $\mathfrak{A}_1=\left(
                                                    \begin{array}{cc}
                                                      \mathcal{A} & B_1C_1 \\
                                                      0 & \mathbf{A} \\
                                                    \end{array}
                                                  \right)
$ with domain $D(\mathfrak{A}_1)=\{(f,g)\in \mathbb{H}\times D(\mathbf{A}):\mathcal{A}f+B_1C_1g\in \mathbb{H}\}$.
By Theorem \ref{general},  the closed-loop system (\ref{2008b}) or (\ref{2008a}) is exponentially stable.
This  avoids the complicated Riesz basis generation procedure in \cite{Guo2008a}.
\end{example}

\begin{example}
The $(w,\widetilde{d})$ part of  \cite[(4.42)]{Feng2017b} is given by
\begin{align}\label{2017a}
    \left\{\begin{array}{l}
w_{tt}(x,t)+w_{xxxx}(x,t)=0,\;\; x\in (0,1), \; t>0, \\
w_{xxx}(0,t)=w_{xx}(0,t)=w(1,t)=0, \; \;   t\ge 0,\\
w_{xx}(1,t)=-c_2w_{xt}(1,t)-c_3w_x(1,t)+\widetilde{d}_{xx}(1,t), \;\;  t\ge 0, \\
\widetilde{d}_{tt}(x,t)+\widetilde{d}_{xxxx}(x,t)=0,\;\; x\in (0,1), \; t>0, \\
\widetilde{d}_{xxx}(0,t)=-c_1\widetilde{d}_t(0,t), \; \;   t\ge 0,\\
\widetilde{d}_{xx}(0,t)=\widetilde{d}(1,t)=\widetilde{d}_x(1,t)=0, \; \;   t\ge 0.\\
\end{array}\right.
\end{align}
Let $l(x,t)=w(1-x,t),\; \widetilde{p}(x,t)=\widetilde{d}(1-x,t)$. Let $c_1,c_2,c_3$, $\mathbf{A}$ and $\mathcal{A}$ be defined as in Example \ref{exam1}. Then,  system (\ref{2017a}) is equivalent to
\begin{align}\label{2017b}
    \left\{\begin{array}{l}
l_{tt}(x,t)+l_{xxxx}(x,t)=0,\;\; x\in (0,1), \; t>0, \\
l_{xxx}(1,t)=l_{xx}(1,t)=l(0,t)=0, \; \;   t\ge 0,\\
l_{xx}(0,t)=c_2l_{xt}(0,t)+c_3l_x(0,t)+\widetilde{p}_{xx}(0,t), \;\;  t\ge 0, \\
\widetilde{p}_{tt}(x,t)+\widetilde{p}_{xxxx}(x,t)=0,\;\; x\in (0,1), \; t>0, \\
\widetilde{p}_{xxx}(1,t)=c_1\widetilde{p}_t(1,t), \; \;   t\ge 0,\\
\widetilde{p}_{xx}(1,t)=\widetilde{p}(0,t)=\widetilde{p}_x(0,t)=0, \; \;   t\ge 0,\\
\end{array}\right.
\end{align}
whose system operator is $\mathfrak{A}_2=\left(
                                                    \begin{array}{cc}
                                                      \mathbf{A} & B_2C_2 \\
                                                      0 & \mathcal{A} \\
                                                    \end{array}
                                                  \right)
$ with domain $D(\mathfrak{A}_2)=\{(f,g)\in \mathbf{X}\times D(\mathcal{A}):\mathbf{A}f+B_2C_2g\in \mathbf{X}\},$
where $B_2=(0,\delta'(x))^T$ and $C_2=(\delta''(x),0)$. By \cite[Theorem 2.1]{Guo2008a}, $B_2$ is admissible for $e^{\mathbf{A}t}$.
It is routine to show that $C_2$ is admissible for $e^{\mathcal{A}t}$.
Therefore, we can apply Theorem \ref{general} to deduce that  system (\ref{2017b}) or (\ref{2017a}) is exponentially stable.
Once again,  we avoid proving the complicated Riesz basis generation procedure in \cite{Feng2017b}.
\end{example}

\begin{example}\label{example33}
In \cite[(2.6) and (3.4)]{Zhou2018}, the coupled system is described by
\begin{align}\label{zhou2018}
    \left\{
      \begin{array}{ll}
       \epsilon_{tt}(x,t)=\epsilon_{xx}(x,t), \\
        \epsilon_x(0,t)=c_2\epsilon_t(0,t)+\widetilde{d}_x(0,t), \\
        \epsilon_x(1,t)=-c_1\epsilon(1,t),\\
\widetilde{d}_{tt}(x,t)=\widetilde{d}_{xx}(x,t), \\
       \widetilde{d}_x(1,t)=-c_0\widetilde{d}_t(1,t), \widetilde{d}(0,t)=0,
      \end{array}
    \right.
\end{align}
where $c_0,c_1,c_2>0$. Let $\mathbb{H}_0 = H^1(0,1)\times L^2(0,1)\}$
with the inner product induced norm $\|(f,g)\|^2_{\mathbb{H}_0} = \int_0^1[|f'(x)|^2+|g(x)|^2]dx+c_1|f(1)|^2$
for any $(f,g)\in \mathbb{H}_0.$
Define
$$
\left\{\begin{array}{l}
\mathcal{A}(f,g)=(g,f''), \crr\disp
 D(\mathcal{A})=\{(f,g)\in \mathbb{H}_0| f'(0)=c_2g(0),f'(1)=-c_1f(1)\};
 \end{array}\right.
$$
and
$$
\left\{\begin{array}{l} \mathbf{A}(f,g)=(g,f''), \crr\disp
 D(\mathbf{A})=\{(f,g)\in \mathbf{X}|\mathbf{A}(f,g)\in \mathbf{X}, f'(1)=-c_0f(1)\}
 \end{array}\right.
$$
 where $\mathbf{X}=H^1_F(0,1)\times L^2(0,1)$, $H^1_F(0,1)=\{f\in H^1(0,1):f(0)=0\}$ with
inner product induce norm $\|(f,g)\|^2_\mathbf{X} = \int_0^1[|f'(x)|^2+|g(x)|^2]dx$
for any $(f,g)\in \mathbf{X}.$ Let
$$
B_0=(0,\delta(x)) \hbox{  and } C_0=(\delta'(x),0)^T.
$$
It is well known that both $\mathcal{A}$ and $\mathbf{A}$ are generators of exponentially stable $C_0$-semigroups on $\mathbb{H}_0$ and $\mathbf{X}$, respectively. By \cite[Theorem 3.1]{Zhou2018}, it follows that $B_0$ is admissible for $e^{\mathbf{A}t}$ and
$C_0$ is admissible for $e^{\mathbf{A}t}$.
The system operator corresponding to (\ref{2008b}) is $\mathfrak{A}_1=\left(
                                                    \begin{array}{cc}
                                                      \mathcal{A} & B_0C_0 \\
                                                      0 & \mathbf{A} \\
                                                    \end{array}
                                                  \right)
$ with domain $D(\mathfrak{A}_1)=\{(f,g)\in \mathbb{H}_0\times D(\mathbf{A}):\mathcal{A}f+B_0C_0g\in \mathbb{H}_0\}$.
By Theorem \ref{general},  the closed-loop system (\ref{zhou2018}) is exponentially stable.
However, Zhou and Guo \cite{Zhou2018} only derived the asymptotic stability of system (\ref{zhou2018}), and
the exponential stability was obtained only in the case that $c_0=1$. Furthermore, by Theorem \ref{general},
the assumption $c_0=1$ in \cite[Theorem 3.2]{Zhou2018} can be removed. Therefore, we improve greatly the main results of \cite{Zhou2018}.
\end{example}

\begin{example}
In \cite[(10)]{Krstic2008}, the coupled system is described as follows
\begin{align}\label{Krstic2008}
    \left\{
      \begin{array}{ll}
       \widetilde{w}_{tt}(x,t)=\widetilde{w}_{xx}(x,t)+(c_1+q)e^{qx}[q\epsilon(0,t)+c_0\epsilon_t(0,t)], \\
        \widetilde{w}_x(0,t)=c_1\widetilde{w}(0,t)-[q\epsilon(0,t)+c_0\epsilon_t(0,t)], \\
        \widetilde{w}_x(1,t)=-c_2\widetilde{w}_t(1,t),\\
\epsilon_{tt}(x,t)=\epsilon_{xx}(x,t), \\
       \epsilon_x(0,t)=c_0\epsilon_t(0,t), \epsilon(1,t)=0,
      \end{array}
    \right.
\end{align}
where $q>0$; $c_0,c_1,c_2$ are the same as in Example \ref{example33}.
Let $w(x,t)=\widetilde{w}(1-x,t),d(x,t)=\epsilon(1-x,t)$. Then we transform (\ref{Krstic2008}) to
\begin{align}\label{Krstic2008a}
    \left\{
      \begin{array}{ll}
       w_{tt}(x,t)=w_{xx}(x,t)+(c_1+q)e^{q(1-x)}[q\epsilon(1,t)+c_0\epsilon_t(1,t)], \\
        w_x(1,t)=-c_1w(1,t)+[qd(1,t)+c_0d_t(1,t)], \\
        w_x(0,t)=c_2w_t(0,t),\\
d_{tt}(x,t)=d_{xx}(x,t), \\
       d_x(1,t)=-c_0d_t(1,t), d(0,t)=0,
      \end{array}
    \right.
\end{align}
Let $\mathbf{A}$ and $\mathcal{A}$ be defined as in Example \ref{example33}. Denote
$$
B=\left(
    \begin{array}{cc}
      (c_1+q)e^{q\cdot} & 0 \\
      0 &  \delta(x-1)\\
    \end{array}
  \right)=(B_1 ,B_2), \hbox{  and } C=\left(
                             \begin{array}{cc}
                               q\delta(x) & c_0\delta(x) \\
                               q\delta(x) & c_0\delta(x) \\
                             \end{array}
                           \right)=\left(
                                     \begin{array}{c}
                                       C_1 \\
                                       C_1 \\
                                     \end{array}
                                   \right)
.
$$
 Since $B_1$ is a linear bounded operator and $B_2$ is admissible for $e^{\mathcal{A}t}$, $B$ is admissible for $e^{\mathcal{A}t}$. By definition and the exponential stability of $e^{\mathbf{A}t}$, there exists a constant $M>0$ such that $|\epsilon(0,t)|\leq \frac{1}{c_1}\|(\epsilon (\cdot,t),\epsilon_t(\cdot,t))\|\leq \frac{M}{c_1}\|(\epsilon (\cdot,0),\epsilon_t(\cdot,0))\|$. Accordingly, $(q\delta(x),0)$ is admissible for $e^{\mathbf{A}t}$.
It is routine to verify that $(0,c_0\delta(x))$ is admissible for $e^{\mathbf{A}t}$. Hence $C$ is admissible for $e^{\mathbf{A}t}$.
The system operator corresponding to (\ref{Krstic2008a}) is $\mathfrak{A}_2=\left(
                                                    \begin{array}{cc}
                                                      \mathcal{A} & BC \\
                                                      0 & \mathbf{A} \\
                                                    \end{array}
                                                  \right)
$ with domain $D(\mathfrak{A}_2)=\{(f,g)\in \mathbb{H}_0\times D(\mathbf{A}):\mathcal{A}f+BCg\in \mathbb{H}_0\}$.
By Theorem \ref{general},  the closed-loop system (\ref{Krstic2008a}) or (\ref{Krstic2008}) is governed by and exponentially stable $C_0$-semigroup.
It is seen that in \cite{Krstic2008}, the authors proved the results by introduce a quite skillful equivalent inner product. Moreover, the proof in \cite{Krstic2008} is lengthy. Here the procedure of constructing skillful inner product
is removed and the proof is greatly simplified.
\end{example}


\begin{thebibliography}{99}

\bibitem{Arendt2001} W. Arendt, C.J.K. Batty, M. Hieber, and F. Neubrander, {\it  Vector-Valued Laplace Transforms and
Cauchy Problems}, Birkh\"{a}user, Basel, 2001.

\bibitem{Engel2000} K.J. Engel and  R. Nagel, {\it One Parameter Semigroups for Linear Evolutional Equations},
 Springer-Verlag, New York,
2000.

\bibitem{Feng2017b} H. Feng and  B.Z. Guo, Active disturbance rejection control: old and new results, {\it Annu. Rev. Control},
44(2017), 238-248.

\bibitem{Guo2007a} B.Z. Guo, C.Z. Xu, The stabilization of a one-dimensional wave equation by boundary feedback with noncollocated observation,
 IEEE Transactions on Automatic Control, 52(2) (2007): 371-377.

\bibitem{Guo2008a} B.Z. Guo, J.M. Wang, and K.Y. Yang, Dynamic stabilization of an Euler-Bernoulli beam under boundary control
and non-collocated observation, {\em Systems Control Lett.}, 57(2008), 740-749.


\bibitem{Krstic2008} M. Krstic, B.Z. Guo, A. Balogh, A. Smyshlyaev, Output-feedback stabilization of an unstable wave equation, Automatica, 44(1) (2008): 63-74.

\bibitem{Weiss1989} G. Weiss, Admissibility of unbounded control operators, {\it SIAM J. Control Optim.}, 27(1989), 527-545.

\bibitem{Weiss1989b} G. Weiss, Admissible observation operators for linear semigroups,
{\it Israel J. Mathematics}, 65(1989),  17-43.

\bibitem{Weiss1994} G. Weiss, Regular linear systems with feedback, Math. Control Signals Systems,
 7 (1994), 23-57.

\bibitem{Zhou2018} H.C. Zhou and B.Z. Guo, Performance output tracking for one-dimensional wave equation with a general disturbance, {\it Eur. J.Control}, 39 (2018), 39-52.

\end{thebibliography}
\end{document}